\documentclass{amsart}
\usepackage{latexsym,amssymb,enumerate, amsmath}
\usepackage{stmaryrd}
\usepackage{latexsym,amssymb,enumerate, amsmath}
\usepackage[dvips]{geometry}
\newenvironment{enumeratei}{\begin{enumerate}[\upshape (1)]}
    {\end{enumerate}}

\newtheorem{theorem}{Theorem}[section]

\newtheorem*{thm1}{Theorem 1}
\newtheorem*{thm3}{Theorem 3}
\newtheorem*{thm4}{Theorem 4}
\newtheorem*{thm5}{Theorem 5}

\newtheorem*{thmG}{Theorem G}
\newtheorem*{prop2}{Proposition 2}

\newtheorem{corollary}[theorem]{Corollary}

\newtheorem{lemma}[theorem]{Lemma}
\newtheorem{proposition}[theorem]{Proposition}

\theoremstyle{definition}

\def\irr#1{{\rm Irr}(#1)}

\begin{document}

\title[Non-trivial intersections of kernels of irreducible characters]{Finite groups with non-trivial intersections of kernels of all but one irreducible characters}

\author[M. Bianchi]{Mariagrazia Bianchi}
\address{Mariagrazia Bianchi, Dipartimento di Matematica F. Enriques,
\newline Universit\`a degli Studi di Milano, via Saldini 50,
20133 Milano, Italy.}
\email{mariagrazia.bianchi@unimi.it}

\author[M. Herzog]{Marcel Herzog}
\address{Marcel Herzog, Schoool of Mathematical Sciences, \newline Tel-Aviv
University, Tel-Aviv, 69978, Israel.} \email{herzogm@post.tau.ac.il}

%\author[E. Pacifici]{Emanuele Pacifici}
%\address{Emanuele Pacifici, Dipartimento di Matematica F. Enriques,
%\newline Universit\`a degli Studi di Milano, via Saldini 50,
%20133 Milano, Italy.}
%\email{emanuele.pacifici@unimi.it}

\thanks{The second author is grateful to the Department of
Mathematics of the University of Milano for its hospitality and
support, while this investigation was carried out. The authors were partially 
supported by the italian INdAM-GNSAGA}


\begin{abstract}
In this paper we consider finite groups $G$ satisfying the following 
condition: $G$ has two columns in its character table which differ by exactly one 
entry. It turns out that such groups exist and they are exactly the finite groups
with a non-trivial intersection of the kernels of all but one irreducible 
characters or, equivalently, finite groups with an irreducible character
vanishing on all but two conjugacy classes. We investigate such groups 
and in particular we characterize their subclass, which properly contains 
all finite groups with non-linear characters of distinct degrees, 
which were characterized by Berkovich, Chillag and Herzog in 1992.
\end{abstract}

\maketitle

\section{Introduction}

It is well known that the intersection of kernels of all irreducible characters
of a finite group is trivial. This gives rise to the following question: which
finite groups, if any, have an intersection of kernels of all but one
irreducible characters which is non-trivial?  

We were lead to this problem by considering an apparently more general question:
which finite groups have two columns in their character table which differ
by exactly one entry? This problem is more general, since if an intersection 
of kernels of all but one
irreducible characters of a finite group is non-trivial and if $b\neq 1$ 
belongs to such an intersection, then clearly the column of the character degrees and the
column of $b$ differ by exactly one entry. The surprising fact is that these
two families of finite groups coincide.

So in this paper we shall investigate finite groups with two columns 
in their character table differing by exactly one entry. Such groups
will be called \emph{$CD1$-groups} (Columns (of the
character table) Differing (by) 1 (entry)). To eliminate trivialities, 
we shall assume that if $G\in CD1$, then $|G|>2$. This problem 
was suggested to us by our late colleague David Chillag. 

All groups in this paper are finite.  The set of all  
irreducible characters of a group $G$ will be denoted by $\irr G$. We shall write
$$\irr G=\{\chi_1,\chi_2,\dots, \chi_k\},$$ 
where $\chi_1=1_G$, the principal character of $G$, and the degrees of the characters
satisfy
$$\chi_1(1)=1\leq \chi_2(1)\leq \dotsb \leq \chi_k(1).$$
If $x\in G$, we shall denote the (conjugacy) class of $x$ in $G$ by $x^G$.

By the orthogonality relations, two columns of the 
character table of a group $G$  cannot be equal. But could
a group $G$, satisfying $|G|>2$,  have two columns           
in its character table differing by exactly one entry, or, in other words, be 
a $CD1$-group? We shall show that such groups
exist and we shall try to determine their structure.

Let $G$ be a $CD1$-group and suppose that the two columns of
the character table of $G$ differing by one entry correspond to the classes
$a^G$ and $b^G$ of $G$. A priori, these classes are not unique. 
However, we shall show in Section 2 that
one of these classes, say $a^G$, must be the unit class $\{1\}$ (hence $a=1$).
Moreover, the two columns 
corresponding to $a^G$ and $b^G$ can 
differ only in the row of
the character $\chi_k$, and the degree of $\chi_k$ must be larger than that
of any other $\chi_i$. In particular, $\chi_k(1)>\chi_{k-1}(1)$.

Denote
$$\chi_i(1)=a_i\quad \text{and}\quad \chi_i(b)=b_i\quad \text{for}\ i=1,2,\dots,k.$$
Then, by the results of Section 2 mentioned above, we have
$$b_i=a_i=\chi_i(1)\quad \text{for}\quad i=1,2,\dots,k-1$$
and 
$$b_k\neq a_k.$$
By the orthogonality relations, $b_k$ is a negative integer satisfying
$$-a_k\leq b_k\leq -1.$$

On the other hand, we shall show that the $k$-th row of the character table of
$G$ is 
$$(a_k,b_k,0,0,\dots,0).$$
Hence the class $b^G$ is unique. Consequently, $G$ is a $CD1$-group
if and only if an intersection of kernels of all but one
irreducible characters of $G$ is non-trivial.

Let us denote
$$N=\bigcap_{i=1}^{k-1} \ker \chi_i.$$
Then $N$ is a normal subgroup of $G$ and $\{1\}\cup b^G\subseteq N$.
The uniqueness of the class $b^G$ implies that
$$N=\{1\}\cup b^G.$$
Therefore $G$ is a $CD1$-group if and only if 
$N\neq \{1\}$.
In particular, an intersection of kernels of all but one 
irreducible characters of  a group $G$ is trivial, except, possibly,
if the missing character's degree is larger than those of the other
characters. 

If  $G$ is a $CD1$-group, then $N$ is the union of two conjugacy classes 
of $G$. Hence $N$
is a minimal normal
subgroup of $G$, which is an elementary abelian $p$-group for some prime $p$.
Thus
$$|N|=|\{1\}\cup b^G|=p^n$$
for some positive integer $n$ and $b$ is of order $p$.

We shall show also in Section 2 that 
$$|C_G(b)|=|P|,$$
where $P$ denotes a Sylow $p$-subgroup of $G$. Hence $C_G(b)=P$ for some
Sylow $p$-subgroup $P$ of $G$ and  
$$C_G(N)=\bigcap_{x\in G}C_G(b^x)=\bigcap_{x\in G}P^x=O_p(G).$$

We have seen above that a  $CD1$-group has an irreducible character which 
vanishes on all but two conjugacy classes. We shall show in Section 2 that 
these two properties are equivalent:
{\it an irreducible character of a group $G$
vanishes on all but two conjugacy classes of $G$ if and only if the columns
coresponding to these classes in the character table of $G$ differ by
exactly one entry.} Moreover, if $G$ is a $CD1$-group, then $(G,N)$ is 
a Camina pair. 

In his paper [6], Stephen Gagola investigated groups which have an irreducible
character which
vanishes on all but two conjugacy classes. Such groups of order greater than
$2$ will be called  {\it
Gagola groups}. The above
observation implies that Gagola groups coincide with the $CD1$-groups.
In particular, our notation applies also to Gagola groups.

Let $G$ be a Gagola group. In his paper Gagola  
completely determined the structure  of the quotient group
$$G/C_G(N)=G/O_p(G).$$
In particular, he proved that if $G$ is
solvable, then a Sylow $p$-subgroup of $G/O_p(G)$ is abelian.
Moreover, if $G$ is solvable, then
$G/O_p(G)$ has a normal $p$-complement, which is
isomorphic to the multiplicative group of a near-field.
The multiplicative groups of  finite near fields  are
in one-to-one correspondence with the class of doubly transitive Frobenius
groups. 

A Frobenius group $G$ is called {\it doubly transitive} if
$$|H|=|F|-1,$$
where $F$ denotes the (Frobenius) kernel of $G$ and $H$ denotes a
(Frobenius) complement of $G$. This implies that
$$|G|=(p^n-1)p^n\quad \text{with} \quad |F|=p^n\ \text{and} \ |H|=p^n-1,$$
where $p$ is a prime and $n$ is a positive integer. Moreover, $F$ is an
elementary abelian $p$-group. We shall use this notation and these fact
throughout this paper.
The finite near-fields were classified by Hans Zassenhaus in [12].

In the non-solvable case, Gagola's result is much more complicated.

The structure of $O_p(G)$ for Gagola groups (or $CD1$-groups) $G$ is 
{\bf an open problem}. Gagola
showed that there is no bound on the derived length or the nilpotence
class of $O_p(G)$.
  
Concerning the structure of the Gagola groups themselves,
Gagola proved the following theorem (see [6], Theorem 6.2).
\begin{thmG} If $G$ is a Gagola group, then $N=C_G(N)$
if and only if $G$ is a doubly transitive Frobenius group.
\end{thmG}

At the end of his paper, Gagola constructs for $p=2$ and $p=3$ 
examples of Gagola groups
which are not $p$-closed. These are important examples of 
non-$p$-closed Camina pairs. 

Our research was
concentrated on the structure of $CD1$-groups themselves, satisfying certain
conditions with respect to the entries $a_k$ and $b_k$.

A $CD1$-group will be called {\it extreme} if 
$$\text{either}\quad b_k=-1\quad \text{or}\quad b_k=-a_k.$$
Our main result in this paper is the following theorem, in which the
extreme $CD1$-groups  are completely determined. Recall that
a group $G$ is called {\it of central type}  if there exists $\chi\in \irr G$
such that
$$[G:Z(G)]=\chi(1)^2.$$

\begin{thm1}
A group $G$ is an extreme $CD1$-group if and only if one
of the following holds:
\begin{enumeratei}
\item $b_k=-a_k$ and $G$ is a $2$-group of central type with $|Z(G)|=2$;
\item $b_k=-1$ and $G$ is a a doubly transitive Frobenius group.
\end{enumeratei}
\end{thm1}

Moreover, we shall prove the following proposition:

\begin{prop2}
The class of extreme $CD1$-groups properly contains
the class of groups with 
non-linear irreducible characters of distinct degrees.
\end{prop2}

Groups with                  
non-linear irreducible characters of distinct degrees ($DD$-groups in short)
were determined by  Berkovich, Chillag
and Herzog in [1]. Proposition 2 implies that Theorem 1
is a proper generalization of
the theorem in [1] (see Theorem 4.2), which itself is a 
proper generalization of Seitz's
theorem in [11] (see Theorem 4.1), in which groups with 
one non-linear irreducible character
were determined. 

Other generalizations of the  Berkovich, Chillag
and Herzog theorem appeared in the literature. We mention
here four such papers: [2] by Berkovich, Isaacs and Kazarin in 1999, 
[9] by Maria Loukaki in 2007, [3] by Dolfi, Navarro and Tiep in
2013 and [4] by Dolfi and Yadav in 2016. 

It follows from the characterization of the $DD$ groups in  [1] that all such
groups are solvable. 
On the other hand, we shall show that the set
of extreme $CD1$-groups contains non-solvable groups, and in particular,
it contains a perfect group. For details, see Section 4,
where Proposition 2 is proved.

In Section 2 we shall also show that 
if $G$ is a $CD1$-group, then
it is of even order and either $|Z(G)|=2$ or the center
of $G$ is trivial. Moreover, if $|Z(G)|=2$, then $Z(G)=\{1,b\}$ and
$b_k=-a_k$.
Hence $CD1$-groups $G$ with  $|Z(G)|=2$ are
extreme $CD1$-groups and Theorem 1 implies the following general result:

\begin{thm3} A group $G$ is a $CD1$-group with $Z(G)\neq 1$ if and only if
it is  a $2$-group of central type with $|Z(G)|=2$.
\end{thm3}

However, the problem of characterizing $CD1$-groups with a
trivial center is {\bf still open}. There exist such groups which are non-extreme.
For example, there exist two groups of order $54$ which are non-Frobenius
$DC1$-groups
with trivial centers and with $a_k=6$ and $b_k=-3$.

Theorem G and Theorem 1 immediately yield the following result.
\begin{thm4} If $G$ is a $CD1$-group, then $N=C_G(N)$
if and only if $b_k=-1$.
\end{thm4}
We were not able to prove this result directly, without using the results 
of Gagola. 
 
Recalling that the element $b$ has prime order $p$, in Section 3 we shall also 
characterize $CD1$-groups $G$ satisfying the 
following conditions:
$p\nmid a_k$ (see Proposition 3.1), $a_k=p^s$ for some positive integer $s$ 
(see  Proposition 3.2)
and $G$ is an $r$-group for some prime $r$ (see Corollary 3.5).

Finally, in Section 5 we shall characterize $CD1$-groups with $a_k$ being  
any prime power
(see Theorem 5).

The structure of this paper is as follows. In Section 2 the basic properties 
of $CD1$-groups will be determined, inluding a classification result 
(see Proposition 2.7). 
In Section 3 we shall
classify $CD1$-groups satisfying $b_k=-1$ and those satisfying $b_k=-a_k$
(see Propositions 3.1 and 3.2, respectively). These results will imply 
Theorem 1. Theorems 3 will be also proved in Section 3.
The relation between
extreme $CD1$-groups and $DD$-groups will be described in Section 4, including 
the proof of Proposition 2. This proposition implies that our Theorem 1
is a proper generalization of the Berkovich-Chillag-Herzog theorem. 
Finally,  Section 5 will be devoted to the
classification of  $CD1$-groups with $a_k$ being a power of a prime 
(see Theorem 5). In particular, $CD1$-groups with $a_k$ being a prime are
determined in Corollary 5.1.

\section{Basic properties of $CD1$-groups}
Recall that a finite group $G$ is called
a {\it $CD1$-group} if $|G|>2$ and two columns of the character table 
of $G$ ($CT(G)$ in short) differ by exactly one entry.

Suppose that these columns correspond to the two 
classes $a^G$ and $b^G$ for some $a,b\in G$ 
and they differ only in row $j$. Thus
$$\chi_i(a)=\chi_i(b) \quad \text{for $i\neq j$}\quad \text{and}\quad \chi_j(a)
\neq \chi_j(b).$$ 
Two such classes will be called
(temporarily) {\it special classes} of $G$. 

We shall also use the following notation: $g=|G|$ and if $u\in G$, then
the column in the $CT(G)$ corresponding to $u^G$
will be denoted by
$U$ and $\chi_i(u)$ will be denoted by $u_i$ for $1\leq i\leq k$. 

In particular,
the column in the $CT(G)$ corresponding to $a^G$
will be denoted by $A$ and $a_i=\chi_i(a)$  for $1\leq i\leq k$. 
Similarly, $B$ will denote the column corresponding to $b^G$ and $b_i=\chi_i(b)$ 
for all $i$. 

Moreover,
if also $v\in G$, then the product of the columns $U$ and $V$ corresponding to
$u^G$ and $v^G$, respectively, is defined as follows:
$$UV=\sum_{i=1}^k \chi_i(u)\chi_i(v)=\sum_{i=1}^k u_iv_i.$$  
First we prove:

\begin{lemma} Let $G$ be a $CD1$-group  and let 
$a^G,b^G$ be special classes of $G$ with $a_j\neq b_j$.
Then we may assume, without loss of generality, that $a^G=\{1\}$ is the 
trivial class and $a_i=\chi_i(1)$ for all $i$. Moreover,
$$j=k,\quad a_k>a_{k-1}\geq 1$$
and $b_k$ is an integer satisfying
$$-a_k\leq b_k\leq -1.$$
We also have
$$g=a_k^2-b_ka_k,$$
$$a_k^2+a_k\leq g\leq 2a_k^2$$
and 
$$\{1\}\cup b^G\subseteq \bigcap_{i=1}^{k-1} \ker \chi_i.$$
\end{lemma}
\begin{proof}
If neither $a$ nor $b$ is equal to $1$, then by the orthogonality relations (OR in short)
$$\sum_{i=1}^k \chi_i(a)\chi_i(1)=\sum_{i=1}^k \chi_i(b)\chi_i(1)=0.$$
It follows, by our assumptions, that $\chi_j(a)\chi_j(1)=\chi_j(b)\chi_j(1)$
and hence $\chi_j(a)= \chi_j(b)$, a contradiction.

So we may assume, without loss of generality, that $a=1$ and $A$
is the column of the degrees $a_i=\chi_i(1)$ in the $CT(G)$. 
Thus $b_i=a_i=\chi_i(1)$ for all $i$ except $i=j$ and consequently
$$\{1\}\cup b^G\subseteq  \bigcap_{\chi_i\in \irr G\setminus \chi_j}\ker \chi_i.$$
Moreover, by the OR,  we have
$$0=BA=\sum_{i=1}^k a_i^2-a_j^2+b_ja_j=g-a_j^2+b_ja_j,$$
so $b_j$ is an integer and
$$g=a_j^2-b_ja_j.$$  
Since $g\neq 1$, it follows that $g>a_j^2$ and hence $-b_j\geq 1$.
But $-a_j\leq b_j\leq a_j$, so 
$$-a_j\leq b_j\leq -1.$$
Thus 
$$a_j^2+a_j\leq g\leq 2a_j^2$$
and $g>2$ implies that $k>2$ and  $a_j>1$. If $a_j\leq a_{k-1}$,
then $g>a_k^2+a_{k-1}^2\geq 2a_j^2$, a contradiction.
Hence
$$j=k\quad \text{and}\quad a_k>a_{k-1}.$$
By replacing $j$ by $k$ in the previous statements,
we obtain the required results. 
\end{proof}

This lemma implies that if $G$ is a $CD1$-group, then
an intersection of kernels of all but one
irreducible characters of $G$ is non-trivial. The converse of this statement
was noticed in the Introduction. Hence the two families of groups coincide.

Moreover, as noticed in the Introduction,
an intersection of kernels of all but one
irreducible characters of  a group $G$ is trivial, except, possibly,
if the missing character's degree is larger than those of the other
characters.
 
We continue with other lemmas, using the notation of Lemma 2.1.
We shall also assume that the first column of the $CT(G)$ corresponds
to the class of $a=1$, i.e. it is the column of the degrees of the characters,
and the second column corresponds to the class of $b$.  In the next lemma
we shall also establish the equivalence between $CD1$-groups and groups 
with an
irreducible character which vanishes on all but two conjugacy classes.

\begin{lemma} (a) Let $G$ be a $CD1$-group. Then
$$\{1\}\cup b^G\subseteq G^{\prime},$$
so
$$|G^{\prime}|\geq 1+|b^G|.$$
Moreover, the $k$-th row of the $CT(G)$ is:
$$(a_k,b_k,0,0,\dots, 0).$$
Hence the non-unit special class is unique.

(b) A group $G$ is a  $CD1$-group if and only if $G$ has an
irreducible character which vanishes on all but two conjugacy classes.
\end{lemma}
\begin{proof}
(a) Since $a_k>1$, $b^G$ lies in the kernel of all linear characters
of $G$. Hence  $b^G$, together with $1$, lies in  $G^{\prime}$
and $|G^{\prime}|\geq 1+|b^G|.$

Moreover, if $d\in G\setminus{(\{1\}\cup b^G)}$, then
$$DA=DB=0 \Rightarrow d_ka_k=d_kb_k.$$
But $b_k\neq a_k$, so $d_k=0$ and hence the $k$-th row of the $CT(G)$
is:
$$(a_k,b_k,0,0,\dots, 0).$$
Therefore, by Lemma 2.1, no class other than $b^G$ can generate together with 
$\{1\}$ a couple of special classes.

(b) By part (a), every $CD1$-group has an
irreducible character which vanishes on all but two conjugacy classes.

Conversely, suppose that the group $G$ has an
irreducible character $\chi $ which vanishes on all but two conjugacy classes: $a^G$ 
and $b^G$ . Then  we may assume, without loss of generality, that $a=1$
and $\chi(a)\neq 0$. By the orthogonality relations for
rows of the $CT(G)$
we have: $\chi(a)+\chi(b)|b^G|=0$ and if $\psi$ is any other 
irreducible character of $G$, then $\chi(a)\psi(a)+\chi(b)\psi(b)|b^G|=0$.
Thus
$\chi(a)\psi(a)=\chi(a)\psi(b)$, which implies that $\psi(a)=\psi(b)$. 
Hence $\chi(a)\neq \chi(b) $ and the columns of $a^G$ and $b^G$ in the $CT(G)$ 
differ by exactly one entry, so
$G$ is a $CD1$-group. 
\end{proof}  
For a $CD1$-group $G$ set
$$N=\bigcap_{i=1}^{k-1} \ker \chi_i.$$
The structure of $N$  will now be determined.
\begin{lemma}  Let $G$ be a $CD1$-group. Then
$$N=\{1\}\cup b^G=\bigcap_{i=1}^{k-1} \ker \chi_i.$$
Thus $N$ is a minimal  normal subgroup of $G$,
the order of $b$ equals a prime $p$ and
$|N|= p^n$
for some positive integer $n$. Hence
$$p^n=|N|=1+|b^G|$$
and 
$$|b^G|=p^n-1\geq p-1.$$ 
In particular, 
$N$ is an elementary abelian
$p$-group, $p\nmid |b^G|$ and 
$$N\leq G^{\prime}.$$
\end{lemma}
\begin{proof}
Since by Lemma 2.2 the class $b^G$ is unique, it follows that
$$N=\{1\}\cup b^G=\bigcap_{i=1}^{k-1} \ker \chi_i.$$
Hence $N$ is a normal subgroup of $G$. Since $N$ is the union
of the two classes $\{1\}$ and $b^G$, it follows that $o(b)=p$ for 
some prime $p$ and $N$ is a minimal  normal subgroup of $G$ of order 
$$|N|=1+|b^G|=p^n\geq p$$
for some positive integer $n$. In particular, $|b^G|=p^n-1\geq p-1$, 
$p\nmid |b^G|$ and $N$ is an elementary abelian
$p$-group. Since by Lemma 2.2 
$$N=\{1\}\cup b^G\subseteq G^{\prime},$$
it follows that $N$ is a (normal) subgroup of $G^{\prime}$.
\end{proof}
The next two lemmas provide important information concerning
$a_k$, $b_k$ and $g$. 
\begin{lemma} Let $G$ be a $CD1$-group. Then
$$ a_k=-b_k|b^G|$$
and
$$g=-a_kb_k(1+|b^G|)=a_k^2\left(\frac {1+|b^G|}{|b^G|}\right)=b_k^2|b^G|(1+|b^G|).$$
Hence $a_k\geq |b^G|\geq p-1$ and $G$ is of even order.
\end{lemma}
\begin{proof}
Since by Lemma 2.2 the $k$-th row of the $CT(G)$ is $(a_k,b_k,0,0,\dots, 0)$,
it follows by the OR that $a_k\cdot 1+b_k\cdot |b^G|=0$. Hence
$ a_k=-b_k|b^G|$ and by Lemma 2.1 we get:
$$g=a_k^2-b_ka_k=-a_kb_k(1+|b^G|)
=a_k^2\left(\frac {1+|b^G|}{|b^G|}\right)=
b_k^2|b^G|(1+|b^G|).$$
Since $-b_k\geq 1$ by Lemma 2.1 and $|b^G|\geq p-1 $ by Lemma 2.3, 
it follows that $a_k\geq |b^G|\geq p-1$
and $g=b_k^2|b^G|(1+|b^G|)$ implies that $G$ is of even order.
\end{proof}

For the next lemma recall that
$b$ is an element of prime order $p$ and $|b^G|=p^n-1$ for a suitable
positive integer $n$.
\begin{lemma} Let $G$ be a $CD1$-group. Then
$$b_k=-p^t$$
for some non-negative integer $t$, implying that
$$a_k=p^t(p^n-1)\quad \text{and}\quad g=p^{n+2t}(p^n-1).$$
In particular, $-b_k$
is the $p$-part of $a_k$ and the following statements hold: $b_k=-1$
if and only if $p\nmid a_k$ and $b_k=-a_k$ if and only if
$a_k=p^r$ for some positive integer $r$. Moreover,
$b_k=-1$ if and only if $g=p^n(p^n-1)$ and $b_k=-a_k$
if and only if $g=2^{1+2t}$
  
\end{lemma}
\begin{proof}
Since by Lemma 2.2 the $k$-th row of the $CT(G)$ is $(a_k,b_k,0,0,\dots, 0)$,
it follows that if $Q$ is a Sylow $q$-subgroup of $G$ for some 
prime $q\neq p$, then $\chi_k$ vanishes on $Q\setminus \{1\}$
and consequently the restriction of $\chi_k$ to $Q$ is a multiple of the regular
character of $Q$ by a positive integer. Thus $|Q|$ divides $a_k=\chi_k(1)$
and it follows that $\frac g{a_k}$ is a power of $p$. Since
$\frac g{a_k}=-b_k(1+|b^G|)=-b_kp^n$, we may conclude that
$$b_k=-p^t$$
for some non-negative integer $t$. Thus 
$$a_k=-b_k|b_G|=p^t(p^n-1),$$
which implies that  $-b_k$
is the $p$-part of $a_k$. Hence 
$b_k=-1$
if and only if $p\nmid a_k$ and $b_k=-a_k$ if and only if
$a_k=p^r$ for some positive integer $r$. Moreover,
$$g=-a_kb_kp^n=p^t(p^n-1)p^tp^n=p^{n+2t}(p^n-1),$$
so if $b_k=-1$ then $g=p^n(p^n-1)$ and if $b_k=-a_k$ then $b_k=b_k(p^n-1)$,
implying that $p^n-1=1$, $p^n=2^1$ and $g=-a_kb_kp^n=2^{2t}2=2^{1+2t}$.
Conversely, if $g=p^n(p^n-1)$, then $2t=0$ and $b_k=-1$ and if 
$g=p^{n+2t}(p^n-1)=2^{1+2t}$, then $n\geq 1$ implies that $p=2$, $p^n=2$ and 
$b_k=-2^t=-a_k$. Hence 
$b_k=-1$ if and only if $g=p^n(p^n-1)$ and $b_k=-a_k$
if and only if $g=2^{1+2t}$.
The proof of Lemma 2.5 is now complete.
\end{proof}
From now on, we shall use the notation $b_k=-p^t$. Recall that a 
$CD1$-group  is called extreme 
if either $b_k=-1$ or $b_k=-a_k$. Therefore a $CD1$-group  is extreme
if either $p\nmid a_k$ or $a_k=p^r$ for some positive integer $r$

Lemma 2.5 implies the following two general results.
\begin{proposition} Let $G$ be a $CD1$-group. Then
the following statements hold:
\begin{enumeratei}
\item $$|C_G(b)|=p^{n+2t},$$  so the centralizer of each non-trivial element of 
$N$ is a 
Sylow $p$-subgroup of $G$. In particular, each non-trivial element of $N$ 
belongs to 
the center of a unique Sylow $p$-subgroup of $G$. 
\item The center of each Sylow $p$-subgroup of $G$ is contained in $N$.
\item If $N$ is not a Sylow $p$-subgroup of $G$, then the Sylow  $p$-subgroups
of $G$ are non-abelian.
\item $$N= \bigcup \{Z(P)\mid P\in Syl_p(G)\}.$$
Moreover, if $P,Q \in Syl_p(G)$ and $P\neq Q$, then $Z(P)\cap Z(Q)=\{1\}$.
\item If $P\in Syl_p(G)$, then $Z(P)=N$ if and only if $P\triangleleft G$.
\item $C_G(N)=O_p(G)$.
\item $O_{p'}(G)=1$.
\item If $x\in G\setminus N$, then $|C_G(x)|=|C_{G/N}(xN)|$. In particular,
$(G,N)$ is a Camina pair.
\item Let $P\in Syl_p(G)$. Then $N_G(Z(P))=N_G(P)$.
\end{enumeratei}
\end{proposition}
\begin{proof}

(1) Since $g=p^{n+2t}(p^n-1)$, it follows that $|C_G(b)|=\frac g{|b^G|}=\frac g{p^n-1}=p^{n+2t}$
and $C_G(b)$ is a Sylow $p$-subgroup of $G$. Since $N=\{1\}\cup b^G$, it follows
that the centralizer of each non-trivial element of $N$ is a Sylow $p$-subgroup of $G$.
In particular, if $y\in N\setminus\{1\}$, then $y$ belongs to the center of some Sylow 
$p$-subgroup $P$
of $G$. Moreover, this Sylow $p$-subgroup $P$ is unique, since if $y$ belongs to 
the center of 
another Sylow $p$-subgroup $Q$
of $G$, then $|C_G(y)|>|P|=p^{n+2t}$, a contradiction. 

(2) On the other hand, if $c\in G\setminus N$, then by Lemma 2.2
$$|C_G(c)|\leq g-a_k^2=p^{n+2t}(p^n-1)-p^{2t}(p^n-1)^2=p^{2t}(p^n-1)(p^n-p^n+1)
=p^{2t}(p^n-1)$$ 
and $|C_G(c)|<p^{n+2t}$. Hence
the center of each Sylow $p$-subgroup of $G$ is contained in $N$. 

(3) Since $N\triangleleft G$, it follows that $N\subseteq P$ for each $P\in Syl_p(G)$. If
$N\subsetneqq P$ for some $P\in Syl_p(G)$, then $N\subsetneqq P$ for 
all $P\in Syl_p(G)$
and by (2) $Z(P)\subsetneqq P$. Hence
if $N$ is not a Sylow $p$-subgroup of $G$, then the Sylow  $p$-subgroups
of $G$ are non-abelian. 

(4) By (1) and (2)  $N= \bigcup \{Z(P)\mid P\in Syl_p(G)\}$. 
If $P,Q \in Syl_(G)$, $P\neq
Q$ and $x\in Z(P)\cap Z(Q)$, then  $x\in N$ and by (1) $x=1$. Hence
$Z(P)\cap Z(Q)=\{1\}$.

(5) Let $P\in Syl_p(G)$. If $Z(P)=N$, then by (4) $Syl_p(G)=\{P\}$ 
and $P\triangleleft G$.
Conversely, if $P\triangleleft G$ then by (4) $N=Z(P)$.

(6) By (1), $C_G(b)=P$ for some $P\in Syl_p(G)$. Hence
$$C_G(N)=\bigcap_{x\in G}C_G(b^x)=\bigcap_{x\in G}P^x=O_p(G).$$

(7) Since $N$ is a normal $p$-subgroup of $G$, it follows that 
$O_{p'}(G)\leq C_G(N)=O_p(G)$. Hence $O_{p'}(G)=1$.

(8) Let $x\in G\setminus N$. Since $N=\bigcap_{i=1}^{k-1} \ker \chi_i$ 
and $\chi_k(x)=0$, 
it follows that $|C_G(x)|=|C_{G/N}(xN)|$. Thus $(G,N)$ satisfies one of the 
definitions of a Camina pair.  

(9) By (4), if $P,Q \in Syl_p(G)$ and $P\neq Q$, then $Z(P)\cap Z(Q)=\{1\}$.
Hence $[G:N_G(Z(P))]=[G:N_G(P)]$, yielding $|N_G(Z(P))|=|N_G(P)|$. But
$N_G(P)\subseteq N_G(Z(P))$, so $N_G(Z(P))=N_G(P)$.
\end{proof}
Notice that if $G$ is an extreme $CD1$-group, then by Theorem 1 
the Sylow $p$-subgroup
of $G$ is normal in $G$. However, for example, the $CD1$-groups 
of order $54$ mentioned 
in the Introduction are $CD1$-groups which satisfy the condition 
$P\triangleleft G$, but they
are non-extreme ($a_k=6$ and $b_k=-3$). Thus the classification of 
$CD1$-groups with $P\triangleleft G$ is still {\bf an open problem}.  

By Lemma 2.3, if $G$ is a $CD1$-group, then $N\leq G^{\prime}$. In the following 
proposition
we shall classify $CD1$-groups which satisfy the equality $N=G^{\prime}$.
\begin{proposition} The group $G$ is a $CD1$-group with 
$$N=G^{\prime}$$
if and only if one of the following statements holds:

(a) $G$ is an extra-special $2$-group of order $2^{2m+1}$.
The degree pattern of $G$ is $(1^{(2^{2m})},2^m)$.

(b) $G$ is a doubly transitive Frobenius groups of order $(p^n-1)p^n$
with a cyclic complement. 
The degree pattern of $G$ is $(1^{(p^n-1)},p^n-1)$.
\end{proposition}
\begin{proof} Suppose, first, that $G$ is a $CD1$-group with
$N=G^{\prime}$. Then $|G^{\prime}|=p^n$ and by Lemma 2.5
$$[G:G^{\prime}]=\frac {p^{n+2t}(p^n-1)}{p^n}=p^{2t}(p^n-1).$$
Therefore $G$ has $p^{2t}(p^n-1)$ linear characters and
$$a_k^2+p^{2t}(p^n-1)=p^{2t}(p^n-1)^2+p^{2t}(p^n-1)=p^{2t}(p^n-1)(p^n-1+1)=p^{n+2t}(p^n-1)=g.$$
Thus  $G$ has only one non-linear irreducible character and by Seitz's theorem 
(see Theorem 4.1 in Section 4)
$G$ satisfies either (a) or (b). 

Conversely, suppose that $G$ is of type (a) or (b). Then, as shown in the proof of 
Proposition 2 in Section 4, $G$ is a $CD1$-group. Moreover, if $G$ is of type (a),
then $G$ is the Sylow $p$-subgroup of itself and by Proposition 2.6.(5) 
$N=Z(G)$. But in 
extra-special groups $G^{\prime}=Z(G)$, so $G^{\prime}=N$, as required. Finally, if 
$G$ is of type (b), then $F=G^{\prime}$ and it follows by Proposition 3.3 that
$N= F=G^{\prime}$, again as required.  The proof of 
Proposition 2.7 is now complete.
\end{proof} 

We continue with information concerning the center $Z(G)$ of a $CD1$-group $G$.
\begin{lemma} Let $G$ be a $CD1$-group. Then
$$Z(G)\leq N$$
and either
$$Z(G)=1,\ |b^G|\geq 2,\ -b_k\leq \frac {a_k}2\ \text{and}\ g=
a_k^2(\frac {1+|b^G|}{|b^G|})\leq \frac 32 a_k^2,$$
or
$$|Z(G)|=2,\ Z(G)=\{1,b\},\ |b^G|=1,\ p=2,\ n=1,\ b_k=-a_k\ \text{and}\ g=2a_k^2.$$
\end{lemma}
\begin{proof}
If $c\in Z(G)$, then $|\chi_k(c)|=a_k$. Since  by Lemma 2.2 the $k$-th row of the 
$CT(G)$ is $(a_k,b_k,0,0,\dots, 0)$, it follows that $c\in \{1\}\cup b^G=N$.
Hence $Z(G)\leq N$.  

If $b\in Z(G)$, then $|b^G|=1$, $Z(G)=\{1,b\}=N$, $|Z(G)|=2=p^n$, 
$a_k=-b_k|b^G|=-b_k$
and  $g=a_k^2(\frac {1+|b^G|}{|b^G|})=2a_k^2$.

If $b\notin Z(G)$, then $Z(G)=1$, $|b^G|\geq 2$, $-b_k=a_k/|b^G|\leq a_k/2$
and $g=a_k^2(\frac {1+|b^G|}{|b^G|})\leq (3/2)a_k^2.$
\end{proof} 

Lemma 2.8 immediately implies the following characterization of  $CD1$-groups $G$
with $|Z(G)|=2$.

\begin{lemma}
Let $G$ be a $CD1$-group. Then $|Z(G)|=2$ if and only if $b_k=-a_k$.
\end{lemma}
Finally, we shall consider $CD1$-groups satisfyng $g=\frac 32a_k^2$.
\begin{lemma} 
Let $G$ be a $CD1$-group satisfying $g=\frac 32a_k^2$. Then
$|N|=p=3$, $b_k=-3^t$, $a_k=2\cdot 3^t$ and $g=2\cdot 3^{2t+1}$.
In particular, such $G$ are $p$-closed.
\end{lemma}
\begin{proof}
By Lemma 2.8, $g=\frac 32a_k^2$ implies that $|b^G|=p^n-1=2$ and $a_k=p^t(p^n-1)=2p^t$.
Hence $|N|=p=3$, $a_k=2\cdot 3^t$, $b_k=-3^t$ and  $g=6p^{2t}=2\cdot 3^{2t+1}$. 
Since  the index of the
Sylow $3$-subgroup in $G$ equals $2$, it follows that $G$ is $p$-closed.
\end{proof}
Notice that if $t=1$, then $g=2\cdot 27=54$ and we get the groups of order $54$
mentioned in the Introduction. 

\section{The basic result and applications}
 
Our main aim in this  section is to prove the following theorem.
\begin{thm1}
A group $G$ is an {\it extreme $CD1$-group} if and only if one
of the following holds:
\begin{enumeratei}
\item $-b_k= 1$ and $G$ is a $2$-group of central type with $|Z(G)|=2$;
\item $-b_k=a_k$ and $G$ is a doubly transitive Frobenius group. 
\end{enumeratei}
\end{thm1}
We shall use the notation of the previous section.
Recall that a $CD1$-group is called {\it extreme} if either $b_k=-1$ or $b_k=-a_k$.
Hence the theorem is a consequence of the following two propositions:
\begin{proposition}
The following statements are equivalent:
\begin{enumeratei}
\item $G$ is a  $CD1$-group with $b_k=-1$,
\item $G$ is a  $CD1$-group with $p\nmid a_k$,
\item $G$ is a doubly transitive Frobenius group  with the Frobenius
kernel $F$,
\end{enumeratei}
where $\{1\}\cup b^G=F$.
\end{proposition}
and
\begin{proposition}
The following statements are equivalent:
\begin{enumeratei}
\item $G$ is a  $CD1$-group with  $b_k=-a_k$,
\item $G$ is a  $CD1$-group with $a_k=p^s$ for some positive integer $s$,
\item $G$ is a $2$-group of central type with $|Z(G)|=2$,
\end{enumeratei}
with $\{1\}\cup b^G=Z(G)$.
\end{proposition}

We first prove the following preliminary
proposition.
\begin{proposition}
Let $G$ be a doubly transitive Frobenius group of order $(p^n-1)p^n$.
Then $G$ is a $CD1$-group with $a_k=p^n-1=|b^G|$, 
$F=N$ and $b_k=-1$.
\end{proposition}
\begin{proof} Since $G$ is a Frobenius group of order $(p^n-1)p^n$ with the
Frobenius
kernel $F$ of order $p^n$, it follows  by Theorem 18.7 in [8] that 
all the irreducible characters of $G$, except one of degree $p^n-1$,
contain $F$ in their kernel. Hence $G$ is a $DC1$-group with $N=F$, $a_k=p^n-1$,  
$|b^G|=p^n-1=a_k$ and by Lemma 2.5 $b_k=-1$. The proof 
of the proposition is complete.
\end{proof}

We proceed now with proofs of the main propositions.
\begin{proof}[Proof of Proposition 3.1]
(1) and (2) are equivalent by Lemma 2.5. We proceed with proving
that (1) and (3) are equivalent.

Suppose, first, that $G$ is a  $CD1$-group with  $b_k=-1$. 
Then by Lemmas 2.3 and 2.5
$g=(p^n-1)p^n$
and $N$ is a minimal normal subgroup of $G$ of order $p^n$.
Moreover, if $x\in N\setminus\{1\}=b^G$, then $|C_G(x)|=\frac {|G|}{|b^G|}=p^n=|N|$
and it follows that $N=C_G(x)$ for each $x\in N\setminus\{1\}$. Thus $G$
is a Frobenius group with $N$ as its kernel, as required (see Feit's
book [5], p.133).

Conversely, suppose that $G$ is a doubly transitive Frobenius group of 
order $(p^n-1)p^n$. Then by Proposition 3.3 
$G$ is a $CD1$-group with $b_k=-1$ and $F=N$, as required.
\end{proof}

\begin{proof}[Proof of Proposition 3.2]
(1) and (2) are equivalent by Lemma 2.5. We proceed with proving
that (1) and (3) are equivalent.

Suppose, first, that $G$ is a  $CD1$-group with $b_k=-a_k$. 
Then by Lemma 2.8
$|Z(G)|=2$ and $g=2a_k^2$. 
Hence 
$$[G:Z(G)]=g/2=a_k^2,$$
so $G$ is a group of central type. Finally, by Lemma 2.5, 
$G$ is a $2$-group, as required.

Conversely, suppose that $G$ is a $2$-group of central type with $|Z(G)|=2$.
Then $[G:Z(G)]=\chi(1)^2$ for some $\chi\in \irr G$ and $|G|=2\chi(1)^2$. Since
$|G|=\sum_{\psi\in \irr G} \psi(1)^2$, it follows that
$$\chi(1)^2=\sum_{\xi\in \irr G\setminus \chi} \xi(1)^2.$$ 

Let $Z(G)=\{1,b\}$. Then $|b|=2$
and for each $\psi\in \irr G$,  $\psi(b)$ is an  integer and $|\psi(b)|=\psi(1)$. 
Hence  $\psi(b)=\pm \psi(1)$.

Since $\sum_{\psi\in  \irr G}\psi(1)\psi(b)=0$, it follows that
$$\chi(1)\chi(b)=-\sum_{\xi\in \irr G\setminus \chi}\xi(1)\xi(b).$$

If $\chi(b)=\chi(1)$, then 
$$\chi(1)^2=-\sum_{\xi\in \irr G\setminus \{\chi\cup 1\}} \xi(1)\xi(b)-1\leq 
\sum_{\xi\in \irr G\setminus \{\chi\cup 1\}} \xi(1)^2-1=
\sum_{\xi\in \irr G\setminus \chi}\xi(1)^2-2,$$
a contradiction.  
So $\chi(b)=-\chi(1)$ and 
$$\chi(1)^2=\sum_{\xi\in \irr G\setminus \chi}\xi(1)\xi(b)=\sum_{\xi\in \irr G\setminus 
\chi} \xi(1)^2.$$
Hence
$\xi(b)=\xi(1)$ for all $\xi\in \irr G\setminus \chi$ and $G$ is a $CD1$-group
with $\chi=\chi_k$, $Z(G)=\{1\}\cup b^G$ and  $b_k=\chi(b)=-\chi(1)=-a_k$, as required.
\end{proof}

As mentioned above, Theorem 1 follows from Propositions 3.1 and 3.2.

We conclude this section with three applications of Propositions 3.1 and 3.2. 

By  Proposition 3.3 doubly transitive Frobenius groups  are  $CD1$-groups with 
$a_k=p^n-1$. Therefore Proposition 3.1  immediately
yields 
\begin{corollary}
$G$ is a  $CD1$-group with $a_k<p$ if and only if 
$a_k=p-1$ and $G$ is a doubly transitive Frobenius group of order $(p-1)p$.
\end{corollary}

If the $CD1$-group $G$ is an $r$-group for some prime $r$, then $p=r$ and 
$a_k=r^s=p^s$ for some positive integer $s$. Therefore Proposition 3.2
immediately yields
\begin{corollary}
Let $G$ be an $r$-group for some prime $r$. Then $G$ is a  $CD1$-group if
and only if $r=2$  and $G$ is a $2$-group of central type with $|Z(G)|=2$.
\end{corollary}
The final result is also an application of Propositions 3.2. 
\begin{thm3} A group $G$ is a $CD1$-group with $Z(G)\neq 1$ if and only if
it is  a $2$-group of central type with $|Z(G)|=2$.
\end{thm3}
\begin{proof} Suppose, first, that $G$ is a $CD1$-group with $Z(G)\neq 1$.
Then by Lemma 2.8 $|Z(G)|=2$ and by Lemma 2.9 $b_k=-a_k$. Thus, by Propositions 3.2, 
$G$ is a $2$-group 
of central type with $|Z(G)|=2$.

Conversely, if $G$ is  a $2$-group of central type   with $|Z(G)|=2$, then by
Proposition 3.2 $G$ is a $CD1$-group with $|Z(G)|\neq 1$,
as required. 
\end{proof}

\section{$DD$-groups vis-a-vis extreme $CD1$-groups}

In [11], Seitz proved the following theorem:
\begin{theorem} The group $G$ has only one non-linear irreducible character
if and only if $G$ is of one of the following two types:
\begin{enumeratei}
\item $G$ is an extra-special $2$-group of order $2^{2m+1}$.
The degree pattern of $G$ is $(1^{(2^{2m})},2^m)$.

\item $G$ is a doubly transitive Frobenius group of order $(p^n-1)p^n$ 
with a cyclic
complement. The degree pattern of $G$ is $(1^{(p^n-1)},p^n-1)$.
\end{enumeratei}
\end{theorem}

In [1], Berkovich, Chillag and Herzog generalized  Seitz's theorem.
They proved
\begin{theorem}  The non-linear irreducible characters of the group
$G$ are of distinct degrees if and only if either $G$ is of one of the two types
{\rm{(a)}} and {\rm{(b)}} of Theorem $4.1$ or  $G$ is the following group:

\noindent{\rm{(c)}} $G$ is the Frobenius group of order $2^33^2$, with the Frobenius
kernel of order $3^2$ and with a quaternion Frobenius
complement of order $2^3$. The degree pattern of $G$ is $(1^{(4)},2,8)$.
\end{theorem}

Groups with non-linear irreducible characters of  distinct degrees 
were called $DD$-groups. It follows immediately from Theorem 4.2 that
$DD$-groups are solvable.
 
The extreme $CD1$-groups were determined in our Theorem 1, which will be restated
now:
\begin{thm1}
A group $G$ is an extreme $CD1$-group if and only if one
of the following holds:
\begin{enumeratei}
\item $b_k=-a_k$ and $G$ is a $2$-group of central type with $|Z(G)|=2$;
\item $b_k=-1$ and $G$ is a a doubly transitive Frobenius group.
\end{enumeratei}
\end{thm1}

Our aim is to show that the set of extreme $CD1$-groups properly contains the set
of $DD$-groups. Thus 
our Theorem 1 is a proper generalization of Theorem 4.2.

\begin{prop2}
The set of extreme $CD1$-groups properly contains
the set of $DD$-groups.
\end{prop2}
\begin{proof} 
First we shall show that the set of $DD$-groups is contained in the
set of extreme $CD1$-groups. Indeed, if $G$ is a $DD$-group, then by Theorem 4.2,
$G$ is either a doubly transitive Frobenius group, 
or it is an extra-special $2$-group.
In the former case $G$ is an extreme $CD1$-group by Proposition 3.3. So it remains 
only to show 
that  extra-special $2$-groups are $2$-groups of central type with $|Z(G)|=2$.
It is well known that an extra-special $2$-group $G$ is of order $2^{2m+1}$
for some positive integer $m$  and satisfies the following two properties:
$|Z(G)|=2$ and $G$ contains an irreducible character of degree $2^m$
(see, for example, [8], Example 7.6(b)). Hence $G$ 
is a $2$-group of central type with $|Z(G)|=2$, as required. The proof
of our claim is complete.

It remains only to show that the set of extreme $CD1$-groups properly 
contains the set
of $DD$-groups. We shall show  that  extreme $CD1$-groups of each
type (1) and (2), as described in Theorem 1, contain groups which are not $DD$-groups.

As an example of type (1), consider the group $G=Q_{16}:C_2$ of order $32$.
The center of $G$ is of order $2$ and it has an irreducible character of 
degree $4$, so $G$ is a $2$-group of central type with $|Z(G)|=2$. By Theorem 1,
$G$ is an extreme $CD1$-group. But $G$ is not extra-special, since $|G'|=4$, so 
in view of Theorem 4.2 $G$ is not a $DD$-group. There are four other groups of order
$32$ with the same property.

The examples of type (2) are more interesting. By 
 the results of
Zassenhaus in [12] (see also [7], Section 20.7) 
there are the following
three non-solvable  Frobenius groups of order $(p^2-1)p^2$ for $p=11,29,59$:  

(i) $G_1=FH$, where the Frobenius kernel $F$ is elementary abelian of order $11^2$
and the Frobenius complement $H$ of order $120$ is isomorphic to $SL(2,5)$. 

(ii) $G_2=FH$, where the Frobenius kernel $F$ is elementary abelian of order $29^2$ 
and the Frobenius complement $H$ of order  $840$ is isomorphic to
$H=SL(2,5)\times C_7$.

(iii) $G_3=FH$, where the Frobenius kernel $F$ is elementary abelian of order $59^2$ 
and the Frobenius complement $H$ of order $3480$ is isomorphic to
$H=SL(2,5)\times C_{29}$.

These groups are non-solvable extreme $CD1$-groups by Theorem 1 and $G_1$ is even
perfect. In view of Theorem 4.2, they are not $DD$-groups.

The proof of Proposition 2 is now complete.

\end{proof}
\section{Charaterizations of $CD1$-groups with $a_k$ a power of a prime}

In this section we shall prove the following classification theorem:
\begin{thm5} Let $r$ be a prime.
Then $G$ is a  $CD1$-group with $a_k=r^s$ for some positive integer $s$
if and only if one of the following cases holds:
\begin{enumeratei}
\item $G$ is a $2$-group of central type with $|Z(G)|=2$ and $|G|=2^{2s+1}$
for some positive integer $s$;
\item  $G$ is a doubly transitive Frobenius group of order $(2^n-1)2^n$,
where $2^n-1=r$ is a Mersenne prime;
\item $G$ is a doubly transitive Frobenius group of order $(3^2-1)3^2=72$;
\item $G$ is a doubly transitive Frobenius group of order $(p-1)p$, 
where $p=2^n+1$ is a Fermat prime. 
\end{enumeratei}
\end{thm5}
\begin{proof}
Suppose, first, that  $r$ is a prime and $G$ is a  $CD1$-group with $a_k=r^s$ 
for some positive integer $s$.
If $r=p$, then (1) holds by Proposition 3.2. 

So suppose that $r\neq p$. Then $p\nmid a_k$ and, by Proposition 3.1,
$G$ is a doubly transitive Frobenius group of order $(p^n-1)p^n$.  
By Proposition 3.3, 
$a_k=|b^G|=p^n-1$ and since $a_k=r^s$, it follows that
$$p^n-r^s=1.$$
Since $p$ and $r$ are primes, Lemma 19.3 in [10] implies
that one of the following three cases holds:

(i) $s=1$, $p=2$ and $r=2^n-1$ is a Mersenne prime;

(ii) $s=3$, $r=2$, $p=3$ and $n=2$;

(iii) $r=2$, $n=1$ and $p=2^s+1 $ is a Fermat prime.

If (i) holds, then $p=2$, $a_k=r=2^n-1$ is a Mersenne prime and
$G$ is a doubly transitive Frobenius group of order $(2^n-1)2^n=r(r+1)$, 
as claimed in (2).

If (ii) holds, then  $r=2$, $a_k=2^3$, $p=3$, $n=2$ and $G$ is a 
doubly transitive Frobenius group
of order $(3^2-1)3^2=72$,
as claimed in (3).

Finally, if (iii) holds, then $r=2$, $n=1$,
$p=2^s+1 $ is a Fermat prime and $G$ is a 
doubly transitive Frobenius group of order
$(p-1)p$,
as claimed in (4).
 
Conversely, suppose that $G$ is a group for which one of 
the cases (1), (2), (3) or (4) holds.
If case (1) holds, then by Proposition 3.2  $G$ is a  $CD1$-group with 
$p=2$ and $a_k=2^s$ 
for some positive integer $s$, as required. 
 
In the other three cases, $G$ is a doubly transitive Frobenius group 
of order $(p^n-1)p^n$, where $p$ is a prime and $p^n-1=r^s$ for some prime $r$ and
some positive integer $s$. Thus it follows by Proposition 3.3 that
that $G$ is a $DC1$-group and $a_k=p^n-1=r^s$, 
as required.

The proof of the theorem is complete.
\end{proof}

The next corollary of Theorem 5 is the last result of this paper. 
We shall denote by $S_3$, $D_8$ and $Q_8$
the following groups: the symmetric group on $3$ letters,
the dihedral group of order $8$ and the quaternion group of order $8$,
respectively.
\begin{corollary}
The group  $G$ is a $CD1$-group with
$$a_k=r$$
for some prime $r$ if and only if either $r=2$ and $G$ is isomorphic to one of the
groups: $S_3$, $D_8$ and $Q_8$, or $r=2^n-1$ is a Mersenne prime and $G$ is a 
doubly transitive Frobenius
group of order $(2^n-1)2^n$.  
\end{corollary}

\begin{proof}
We need to determine all  groups with  $s=1$ in Theorem 5.

In case (1) of Theorem 5, $r=2$  and the assumption that $s=1$ 
implies that $G$ is a $2$-group
of central type of order $8$. Hence $G$ is either $D_8$ or $Q_8$.

In case (2), it follows by Proposition 3.3 that $a_k=2^n-1=r$, 
a  Mersenne prime. Therefore
all groups in (2) satisfy our assumption.

In case (3), it follows by Proposition 3.3 that $a_k=3^2-1=8$, so 
$G$ does not satisfy our
assumption.

Finally, in case (4), it follows by Proposition 3.3 that $a_k=p-1=2^n$. Therefore $r=2$
and $s=1$ implies that $n=1$. Hence $p=3$ is a Fermat prime and $G$ is a Frobenius group 
of order  $6$. Thus $G$ is $S_3$.

The proof of the corollary is complete.

\end{proof}

\end{document}